\documentclass{article}
\pdfoutput=1
\usepackage[margin=1in]{geometry}
\usepackage{amsmath, amsfonts, amsthm}
\usepackage{tikz}
\usetikzlibrary{matrix}

\newtheorem{defn}{Definition}
\newtheorem{lemma}[defn]{Lemma}
\newtheorem{prop}[defn]{Proposition}
\newtheorem{thm}[defn]{Theorem}
\newtheorem{cor}[defn]{Corollary}

\newcommand{\ad}{\operatorname{ad}}
\newcommand{\GL}{\operatorname{GL}}
\newcommand{\Fbar}{\overline{F}}
\newcommand{\brak}[1]{\langle #1\rangle}
\newcommand{\eps}{\varepsilon}
\newcommand{\CC}{\mathbb{C}}
\newcommand{\NN}{\mathbb{N}}

\title{Fiber Bundles and Parseval Frames}
\author{DEVANSHU AGRAWAL  and JEFF KNISLEY}
\date{December 12, 2015}

\begin{document}
\maketitle

\begin{abstract} Continuous frames over a Hilbert space have a rich and sophisticated structure that can be represented in the form of a fiber bundle.  The fiber bundle structure reveals the central importance of Parseval frames and the extent to which Parseval frames generalize the notion of an orthonormal basis.   \end{abstract}

\section{Introduction}

Let $H$ be a separable Hilbert space and let $\left( X,\mu \right) $ be a measure space where the cardinality of $X$ is at least the dimension of $H.$ A frame on $H$ is a map $f:X\mapsto H$ for which $\left\langle \phi ,f\left( x\right) \right\rangle$ is weakly measurable with respect to $\mu $ for all $\phi \in H$ and for which there exist constants $0<A\leq B$ such that 
$$
A\left\Vert \phi \right\Vert ^{2}\leq \int_{X}\left\vert \left\langle \phi,f\left( x\right) \right\rangle \right\vert ^{2}d\mu \left( x\right) \leq B\left\Vert \phi \right\Vert ^{2} 
$$
The constants $A,B$ are called the \emph{frame bounds} of $f$, and the frame $f(x)$ is often referred to as a \emph{continuous frame} \cite{ali1993continuous, heil1989continuous, rahimi2006continuous}. If $A=B=1,$ then $f$ is called a Parseval frame, and it follows immediately that every orthonormal basis of $H$ is a Parseval frame \cite{ali2014coherent}.  

Every orthonormal basis is unitarily equivalent to every other orthonormal basis. Moreover, if $f\left( x\right)$ is a Parseval frame and $U\in L\left( H\right)$ is unitary, then 
$g\left( x\right) =U\left[ \ f\left( x\right) \right] $ 
is also a Parseval frame. However, it is possible to have two Parseval frames in a Hilbert space that are not unitarily equivalent. For example, let $X=[0,d)$ under Lesbesgue measure, where $d$ is the dimension of $H$ and may be infinite, and let $\left\{ h_{k}\right\} _{k=0}^{d-1}$ be an orthonormal basis for $H.$ Then the $H$ valued function 
$$
F\left( x\right) =e^{i2\pi x}\ h_{_{\left\lfloor x\right\rfloor }}, 
$$
where $\left\lfloor \; \right\rfloor $ is the floor function, is a frame on $H$.  Any unitary operator $U$ such that 
$$
U\left[ F\left( x\right) \right] = h_{\left\lfloor x\right\rfloor }
$$
would necessarily imply $U\left[ F\left( k\right) \right] =h_{k},$ thus implying that $U=I.$ However, 
$$U\left[ F\left( k+\frac{1}{2}\right) \right] =-h_{k}$$
 implies that $U\neq I,$ and consequently, $F\left( x\right) $ is not unitarily equivalent to $\left\{ h_{k}\right\} .$

Indeed, Parseval frames over $H$ appear to have a rich and sophisticated structure. For example, if both $\left\{ f_{k}\right\} _{k=0}^{d-1}$ and $\left\{ h_{k}\right\} _{k=0}^{d-1}$ are orthonormal bases of $H,$ and if we define 
$$
F\left( x\right) =\cos \left( 2\pi x\right) f_{\left\lfloor x\right\rfloor}+i\sin \left( 2\pi x\right) h_{\left\lfloor x\right\rfloor },
$$
then for any $\phi \in H,$ we notice that 
$$
\int_{0}^{d}\left\vert \left\langle \phi ,F\left( x\right) \right\rangle
\right\vert ^{2}dx=\sum_{k=0}^{d-1}\int_{k}^{k+1}\left\vert \left\langle
\phi ,F\left( x\right) \right\rangle \right\vert
^{2}dx=\sum_{k=0}^{d-1}\int_{0}^{1}\left\vert \left\langle \phi ,F\left( \xi
+k\right) \right\rangle \right\vert ^{2}d\xi 
$$
and 
$F\left( \xi +k\right) =\cos \left( 2\pi \xi \right) f_{k}+i\sin \left(2\pi \xi \right) h_{k},$
from which it follows that $F\left( x\right) $ is also a Parseval frame. Moreover, if $F\left( x\right) $ and $G\left(x\right) $ are both Parseval frames over $H,$ then define 
$$
K\left( x,y\right) = \cos \left( 2\pi y\right) F\left( x\right) +i\sin \left(
2\pi y\right) G\left( x\right) 
$$%
for $y$ in $[0,1].$ The same approach as above shows that $K\left(x,y\right)$ over $X\times Y$ is also a Parseval frame, and in fact, the same approach leads to the following proposition:

\begin{prop}
If $F_{j}\left( x\right) $, $j\in \mathbb{Z}$, are Parseval frames from a measure space $\left( X,\mu \right) $ into a separable Hilbert space $H,$ and if 
$\alpha _{j}\left( y\right) ,\ \ j\in \mathbb{Z}$, 
is an orthonormal subset of $L^{2}\left( Y,\nu \right) $ for a measure space $\left( Y,\nu \right) ,$ then 
$$
K\left( x,y\right) =\sum_{j=-\infty }^{\infty }\alpha _{j}\left( y\right) F_{j}\left( x\right) 
$$
is a Parseval frame on $H$ over the product measure space of $X$ and $Y.$  
\end{prop}

As these examples illustrate, the richness of the structure of frames in general and Parseval frames in particular is due to a dependence on an index $x$ that does not change linearly with respect to $H$ itself. However, these examples also illustrate that frames are related to transformations that are linear in a larger vector space containing $H.$ 

For an $m$-dimensional Hilbert space $\mathbb{C}^{m}$ and a finite index set $X=\left\{ 1,\ldots ,n\right\} ,$ where $m\leq n,$ both of these concepts are well understood. Given 
$\left\{ f_{1},\ldots ,f_{n}\right\} \subset \mathbb{C}^{m},$ 
we form the matrix 
$$
T=\left[  \begin{array}{c}
              f_{1}^{\ast } \\ 
              \vdots        \\ 
              f_{n}^{\ast }
          \end{array}
\right] 
$$
where $f_{j}^{\ast }$ is the Hermitian transpose of the column vector $f_{j}.$ The singular value decomposition of $T$ is given by 
$$
T=U\tilde{\Sigma}V^{\ast }\ \ where\ \ \tilde{\Sigma}=
\left[  \begin{array}{c}
                \Sigma  \\ 
                  0
        \end{array}
\right] 
$$
with $\Sigma =diag\left( s_{1},\ldots ,s_{m}\right);$ where $s_{1}\geq \ldots \geq s_{m}\geq 0$ are the singular values of $T;$ where $U:\mathbb{C}^{n}\mapsto \mathbb{C}^{n}$ is unitary; and where $V:\mathbb{C}^{m}\mapsto \mathbb{C}^{m}$ is also unitary. \ For any $\phi \in \mathbb{C}^{m},$ it follows that
$$
\left\Vert T\phi \right\Vert ^{2}=\sum_{k=1}^{n}\left\vert \left\langle \phi, f_{k}\right\rangle \right\vert ^{2}
$$
and from the properties of the singular value decomposition, 
$$s_{m}^{2}\left\Vert \phi \right\Vert ^{2}\leq \left\Vert T\phi \right\Vert^{2}\leq s_{1}^{2}\ \left\Vert \phi \right\Vert ^{2}$$
If $s_{m}>0,$ then $\left\{ f_{1},\ldots ,f_{n}\right\} $ is a frame on $\mathbb{C}^{m}$ with frame bounds $A=s_{m}^{2}$ and $B=s_{1}^{2}$  \cite{casazza2013introduction,aldroubi2007optimal}. Moreover, $S=T^{\ast }T$ is a positive matrix called the \emph{frame operator} for $\left\{ f_{1},\ldots ,f_{n}\right\} $. 

If we define 
$$g_{j}=S^{-1/2}f_{j}$$
then $\left\{ g_{1},\ldots ,g_{n}\right\} $ is a \emph{Parseval frame} in $\mathbb{C}^{m},$ and in fact, if we define 
$$
T_{p}=\left[ \begin{array}{c}
               g_{1}^{\ast } \\ 
               \vdots        \\ 
               g_{n}^{\ast }%
            \end{array}
\right] 
$$
then the singular value decomposition of $T_{p}$ is 
$$
T_{p}=U\left[ \begin{array}{c}
                         I_{m} \\ 
                          0
              \end{array}
\right] V^{\ast } 
$$
where $U,V$ are the same as for $T$ and $I_{m}$ is the $m\times m$ identity matrix. It follows that every Parseval frame $\left\{ g_{1},\ldots, g_{n}\right\} $ over $\mathbb{C}^{m}$ is of the form 
$$
\left[ \begin{array}{c}
          g_{1}^{\ast } \\ 
          \vdots        \\ 
          g_{m}^{\ast } \\ 
          \vdots        \\ 
          g_{n}^{\ast }
       \end{array}%
\right] =
\left[ \begin{array}{ccc}
             u_{11} & \ldots  & u_{1m} \\ 
            \vdots  & \ddots  & \vdots \\ 
             u_{m1} & \ldots  & u_{mm} \\ 
            \vdots  & \ddots  & \vdots \\ 
             u_{n1} & \ldots  & u_{nm}
        \end{array}
\right] 
\left[ \begin{array}{ccc}
            v_{11} & \ldots  & v_{1m}  \\ 
           \vdots  & \ddots  & \vdots  \\ 
            v_{m1} & \ldots  & v_{mm}
       \end{array}
\right] 
$$
where the first matrix is the first $m$ columns of $V^{\ast }.$ This is essentially the same as the construction in \cite{casazza2008classes} and also in \cite{casazza2001uniform}.  Moreover, $T_{p}$ is a partial isometry, which implies that a Parseval frame is necessarily the orthogonal projection onto $\mathbb{C}^{m}$ of an orthonormal basis of a containing Hilbert space $\mathbb{C}^{n}$ (see for instance \cite{casazza2007generalization,casazza2013introduction, balan2005excess}).   

The examples above and the characterization of the finite dimensional cases implies the desirability of a top down approach to frames.  In particular, we use orbit spaces and fiber bundles both to establish the central importance of Parseval frames and to indicate how the structure of the set of all Parseval frames over a Hilbert space can be effectively represented and studied. 

The paper is organized as follows: In Section 2, we begin by reviewing the definition and basic properties of frames important to this paper. In section 3, we show that the set of all frames on a Hilbert space can be fibrated into orbits of a certain group action, thereby reducing the study of frames to only Parseval frames. In section 4, we establish the fiber bundle structure of the collection of all Parseval frames over a Hilbert space corresponding to a given index set $X$.   In section 5, we provide some basic examples  of this fiber bundle structure along with some possible applications and future directions.  In subsequent papers, we extend the ideas in this first paper to Gabor frames and the discretization problem.  

\section{Properties of Continuous Frames}

We begin with some basic properties of continuous frames. The definitions and properties were obtained from or inspired by \cite{ali2014coherent}. For the remainder of the paper, let $H$ be a complex Hilbert space with the inner product $\brak{\cdot,\cdot}$ linear in the first argument. Let the indexing set $X$ be a locally compact space with positive Borel measure $\mu.$  

The definition of a continuous frame was stated at the outset, but we repeat it here for completeness. For details, see \cite{ali2014coherent, ali1993continuous, heil1989continuous, rahimi2006continuous}:

\begin{defn}
A frame on $H$ is a map $f:X\mapsto H$ for which $\left\langle \phi, f\left( x\right) \right\rangle$ is weakly measurable with respect to $\mu$ for all $\phi \in H$ and for which there exist constants $0<A\leq B$ such that 
$$
A\left\Vert \phi \right\Vert ^{2}\leq \int_{X}\left\vert \left\langle \phi,f\left( x\right) \right\rangle \right\vert ^{2}d\mu \left( x\right) \leq B\left\Vert \phi \right\Vert ^{2} 
$$
\end{defn} 
\noindent The constants $A,B$ are called the \emph{frame bounds} of $f$, and if $A=B=1,$ then $f$ is called a \emph{Parseval frame}. Specifically, a Parseval frame $f$  satisfies 
$$\int_{X}\left\vert \left\langle \phi,f\left( x\right) \right\rangle \right\vert ^{2}d\mu \left( x\right) = \left\Vert \phi \right\Vert ^{2} $$
for all $\phi \in H$.  Note that the range of $f(X)$ need not be a linearly independent set even if $X$ is countable. Note also that $f:X\mapsto H$ is not required to be injective.  

To each frame is associated operators that describe the decomposition and reconstruction of a vector with respect to the frame. These operators are introduced in the following proposition.
\begin{prop} \label{Operators}
Let $f:X\mapsto H$ be a frame. The operator defined by
\begin{equation}
V:H\mapsto L^2(X), \quad
(V\phi)(x) = \brak{\phi, f(x)}_H
\tag{Analysis Operator}
\end{equation}
is well-defined, bounded, and has a bounded inverse. The adjoint of $V$ is given by
\begin{equation}
V^*:L^2(X)\mapsto H, \quad
V^*\alpha = \int_X \alpha(x)f(x)\,d\mu(x),
\tag{Synthesis Operator}
\end{equation}
where convergence of the integral is taken in $H$. Finally, the operator defined by
\begin{equation} \label{S}
S:H\mapsto H, \quad
S\phi = V^*V\phi = \int_X\brak{\phi, f(x)}_H f(x)\,d\mu(x)
\tag{Frame Operator}
\end{equation}
is positive, self-adjoint, bounded, and has a bounded inverse.
\end{prop}

The analysis operator maps a vector in $H$ to a function in $L^2(X)$ describing the projections of the vector on the frame elements. 
The following theorem establishes a deeper connection between a frame and its frame operator.

\begin{thm} \label{FrameS}
A map $f:X\mapsto H$ is a frame if and only if the operator $S:H\mapsto H$ (given by Eq. \ref{S}) is bounded and has bounded inverse. If $f$ is a frame, then its frame bounds are given by $A=\frac{1}{\|S^{-1}\|_H}$ and $B=\|S\|_H$.
\end{thm}
\noindent It follows that a map $f:X\mapsto H$ is a Parseval frame if and only if its frame operator is the identity operator $S=I$.

The following theorem shows that the frame condition is sufficient for the existence of a dual frame and hence for the reconstruction of a vector given its projections on the frame elements.
\begin{thm} \label{DualFrame}
Let $f:X\mapsto H$ be a frame. Then, there exists a frame $\tilde{f}:X\mapsto H$ such that for all $\phi\in H$,
\begin{equation}
\phi = \int_X\brak{\phi, f(x)} \tilde{f}(x)\,d\mu(x).
\tag{Reconstruction Property}
\end{equation}
Any such $\tilde{f}$ is called a \emph{dual frame} of $f$. Moreover, if $f$ is the frame operator of $f$, then the map $\tilde{f}:X\mapsto H$ defined by $\tilde{f}(x)=S^{-1}f(x)$ is a dual frame of $f$.
\end{thm}

The canonical dual frame of $f(x)$ is $\tilde{f}(x) = S^{-1}f(x)$. It is also sometimes referred to as ``the dual frame.''  If $S$ is the frame operator of $f$, then $S^{-1}$ is the frame operator of the dual frame $\tilde{f}$. Consequently, $f$ is the dual frame of $\tilde{f}$, so that for all $\phi\in H$,
\begin{equation}
\phi = \int_X\brak{\phi, \tilde{f}(x)} f(x)\,d\mu(x)
= \int_X\brak{\phi, S^{-1}f(x)} f(x)\,d\mu(x).
\end{equation}
It is important to note, however, that there are dual frames for $f$ other than the canonical dual. Also, in what follows, we will often use a subscript of $H$ to denote inner products or norms with respect to the initial Hilbert space $H$. 

\section{The Orbit Space of Frames}

In this section, we consider the set of all frames on the Hilbert space $H$. In particular, we will show that this set may be fibrated into orbits under the action of linear deformations. We will also show that every frame may be linearly deformed or ``projected'' to a Parseval frame, just as how every basis may be linearly deformed into an orthonormal basis.

Let $\GL(H)$ be the group of all invertible bounded linear operators on $H$ with bounded inverse. Let $\GL^+(H)\subset \GL(H)$ be the cone of all positive operators in $\GL(H)$. We would like to establish some properties about $\GL^+(H)$ as it relates to $\GL(H)$. In particular, we establish the concept of the polar decomposition of an operator in a ``bundle friendly''  manner. 

First, we need a definition.
\begin{defn}
For every $A\in \GL(H)$, let the map $\ad_A:\GL(H)\mapsto \GL(H)$ be defined $$\ad_A(B) = ABA^*.$$ We say that $\ad_A(B)$ is the \emph{adjugation} of $B$ by $A$.
\end{defn}

We define the relation $\sim$ on $\GL(H)$ by $B\sim B^\prime$ if and only if $\ad_A(B)=B^\prime$ for some $A\in\GL(H)$. 

\begin{prop}
The relation $\sim$ is an equivalence relation.
\end{prop}
\begin{proof}
Let $B\in\GL(H)$. Clearly, $\ad_I(B)=IBI^*=B$, so that $B\sim B$. Thus, $\sim$ is reflexive. Given $B,B^\prime\in\GL(H)$, suppose $B\sim B^\prime$. That is, $ad_A(B)=ABA*=B^\prime$ for some $A\in\GL(H)$. Then,
\[ \ad_{A^{-1}}(B^\prime)
= A^{-1}B^\prime (A^{-1})^*
= A^{-1}ABA^*(A^*)^{-1}
= B. \]
Thus, $B^\prime\sim B$, and hence $\sim$ is symmetric. Finally, suppose $B\sim B^\prime$ and $B^\prime\sim B^{\prime\prime}$. Thus, $\ad_A(B)=ABA^*=B^\prime$ and $\ad_{A^\prime}(B^\prime)=A^\prime B^\prime (A^\prime)^*=B^{\prime\prime}$ for some $A,A^\prime\in\GL(H)$. Then,
\[ \ad_{A^\prime A}(B)
= A^\prime AB(A^\prime A)^*
= A^\prime (ABA^*)(A^\prime)^*
= A^\prime B^\prime (A^\prime)^*
= B^{\prime\prime}. \]
Thus, $B\sim B^{\prime\prime}$, and hence $\sim$ is transitive. 
\end{proof}

The equivalence classes in $\GL(H)$ induced by $\sim$ are called \emph{adjugacy classes}.

\noindent We now have the following result.

\begin{prop}
The space $GL^+(H)$ is an adjugacy class in $\GL(H)$.
\end{prop}
\begin{proof}
Let $A\in\GL(H)$ and $B\in\GL^+(H)$. For any $\phi\in H$, we have
\[ \brak{\ad_A(B)\phi, \phi}
= \brak{ABA^*\phi, \phi}
= \brak{BA^*\phi, A^*\phi}\geq 0, \]
where the inequality holds since $A^*\phi\in H$ and $B$ is positive. Thus, $\ad_A(B)$ is positive, so that $GL^+(H)$ is closed under adjugation. Let $B,C\in\GL^+(H)$. Since $B,C$ are positive, then $B=SS^*$ and $C=TT^*$ for some $S,T\in\GL(H)$. There exists $A\in\GL(H)$ such that $T=AS$. We have
\[ \ad_A(B) = ABA^*
= ASS^*A^*
= (AS)(AS)^*
= TT^*
= C. \]
Thus, $B\sim C$. Ergo, $\GL^+(H)$ is an adjugacy class.
\end{proof}

\noindent Because frame operators are elements of $\GL^+(H)$, the above proposition will be useful in subsequent discussions of frame operators.

Next, we define the projection 
\begin{equation}
\rho:\GL(H)\mapsto \GL^+(H), \quad
\rho(A)=AA^*,
\end{equation}
which is important in the proof of the following lemma establishing an important relationship between $\GL^+(H)$ and $\GL(H)$.

\begin{prop}
If the unitary group $U(H)$ acts on $GL(H)$ by right multiplication, then the orbit space $\GL(H)/U(H)$ is in one-one correspondence with $GL^+(H)$.
\end{prop}
\begin{proof}
Since every positive operator $B\in\GL^+(H)$ can be written in the form $B=AA^*$ with $A\in\GL(H)$, then $\rho$ is surjective. Let $A\in\GL(H)$ and $U\in U(H)$. We have
\[ \rho(AU) = (AU)(AU)^*
= AUU^*A^*
= AIA^*
= AA^*
= \rho(A). \]
In particular, we have
\begin{align*}
\ker(\rho) &= \{A\in\GL(H): \rho(A)=I\} \\
&= \{A\in\GL(H): AA^*=I\} \\
&= U(H).
\end{align*}
Since $\GL(H)/\ker(\rho)$ is in one-one correspondence with $\rho(\GL(H))=\GL^+(H)$, the quotient space follows.
\end{proof}

Therefore, $\GL(H) = \GL^+(H)U(H)$, which is to say that every operator in $\GL(H)$ can be factored into a positive operator in $\GL^+(H)$ and a unitary operator in $U(H)$. This is simply the polar decomposition of an operator.

The above discussion of $\GL(H)$ is important because we are interested in linear deformations of frames. We begin by introducing spaces of frames over an index set $X$.  We first define the Banach space $J$ to be
\begin{equation}
J = L^\infty(X, H)
= \left\{f:X\mapsto H: \sup_{x\in X} \|f(x)\|_H <\infty\right\},
\end{equation}
equipped with the norm
\[ \|f\|_J = \sup_{x\in X} \|f(x)\|_H. \]
where the subscripts indicate the space in which the norm is applied. In addition, we define the spaces
\begin{align}
F &= \{f\in J: f \mbox{ is a frame}\} \\
F_0 &= \{f\in F: f \mbox{ is Parseval}\}.
\end{align}
Note that $F$ is restricted to frames whose frame elements have uniformly bounded norms. Moreover, since $H$ is separable, it follows that $F_0$ is non-empty. That is, $H$ has at least one Parseval frame. 

The fibration of $F$ will be given by the action of $\GL(H)$. The following result establishes that this action is continuous.
\begin{lemma}
Let $A\in\GL(H)$ and $f\in J$ and define the action of $A$ on $f$ by
\[ (Af)(x) = A[f(x)]. \]
Then $\GL(H)\subset \GL(J)$. That is, $A$ and $A^{-1}$ are bounded on $J$.
\end{lemma}
\begin{proof}
We have the operator norm
\[ \|A\|_J = \sup_{f\in J,f\neq 0} \frac{\|Af\|_J}{\|f\|_J}
= \sup_{f\in J,f\neq 0}\sup_{x\in X} \frac{\|Af(x)\|_H}{\|f\|_J}. \]
Since $A\in\GL(H)$, then we have
\begin{align*}
\sup_{f\in J,f\neq 0}\sup_{x\in X} \frac{1}{\|A^{-1}\|_H}\frac{\|f(x)\|_H}{\|f\|_J} \leq &\|A\|_J \leq \sup_{f\in J,f\neq 0}\sup_{x\in X} \|A\|_H \frac{\|f(x)\|_H}{\|f\|_J} \\
\frac{1}{\|A^{-1}\|_H} \sup_{f\in J,f\neq 0}\sup_{x\in X}\frac{\|f(x)\|_H}{\|f\|_J} \leq &\|A\|_J \leq \|A\|_H\sup_{f\in J,f\neq 0}\sup_{x\in X}\frac{\|f(x)\|_H}{\|f\|_J} \\
\frac{1}{\|A^{-1}\|_H}\sup_{f\in J,f\neq 0} \frac{\|f\|_J}{\|f\|_J} \leq &\|A\|_J \leq \|A\|_H\sup_{f\in J,f\neq 0} \frac{\|f\|_J}{\|f\|_J} \\
\frac{1}{\|A^{-1}\|_H} \leq &\|A\|_J \leq \|A\|_H.
\end{align*}
Therefore, $A\in\GL(J)$.
\end{proof}

Define the ``frame operator map'' $S:F\mapsto \GL^+(H)$ such that $S(f)$ is the frame operator of $f$. The following lemma is of central importance.
\begin{lemma} \label{LinearDeform}
Let $f\in F$ and $A\in\GL(H)$. Then $Af\in F$ and the frame operator of $Af$ is 
$$S(Af) = \ad_A(S(f)) = AS(f)A^*.$$
\end{lemma}
\begin{proof}
Let $\phi\in H$. Then, $A^*\phi\in H$. Since $f$ is a frame, then
\[ a\|A^*\phi\|_H^2 \leq \int_X\lvert\langle A^*\phi, f(x)\rangle_H\rvert^2\,d\mu(x) \leq b\|A^*\phi\|_H^2. \]
Since $\langle A^*\phi,f(x)\rangle_H = \langle\phi, A[f(x)]\rangle_H = \langle\phi, (Af)(x)\rangle_H$, then
\[ a\|A^*\phi\|_H^2 \leq \int_X\lvert\langle \phi, (Af)(x)\rangle_H\rvert^2\,d\mu(x) \leq b\|A^*\phi\|_H^2. \]
Since $A\in\GL(H)$ and since $A$ and $A^*$ have the same norms, then we have
\[ \frac{a}{\|A^{-1}\|_H^2}\|\phi\|_H^2 = a\left(\frac{\|\phi\|_H}{\|A^{-1}\|_H}\right)^2
\leq a\|A^*\phi\|_H^2. \]
We also have
\[ b\|A^*\phi\|_H^2 \leq b(\|A\|_H\|\phi\|_H)^2 = b\|A\|_H^2\|\phi\|_H^2. \]
Therefore, for all $\phi\in H$, we have
\[ \frac{a}{\|A^{-1}\|_H^2}\|\phi\|_H^2 \leq \int_X\lvert\phi,(Af)(x)\rangle_H^2\,d\mu(x) \leq b\|A\|_H^2\|\phi\|_H^2. \]
Thus, $Af$ is a frame.

The frame operator of $Af$ is given by
\begin{align*}
S(Af)\phi &= \int_X\langle\phi, (Af)(x)\rangle_H(Af)(x)\,d\mu(x) \\
&= \int_X\langle\phi,A[f(x)]\rangle_HA[f(x)]\,d\mu(x) \\
&= \int_X A\langle A^*\phi, f(x)\rangle_H f(x)\,d\mu(x).
\end{align*}
Since $A$ is bounded and hence uniformly continuous, then
\begin{align*}
S(Af)\phi &= A\int_X\langle A^*\phi, f(x)\rangle_H f(x)\,d\mu(x) \\
&= AS(f)A^*\phi.
\end{align*}
Therefore, $S(Af) = \ad_A(S)$.
\end{proof}

The set of frames $F$ can therefore be fibrated into orbits under the action of $\GL(H)$. We let $F/\GL(H)$ denote the resulting space of orbits. Note that since $H$ is a complex Hilbert space, the group $\GL(H)$ is topologically connected. As a consequence, the orbits in $F/\GL(H)$ are connected spaces in $J$.

Because all orthonormal bases in $H$ are unitarily equivalent, exactly one orbit in $F/\GL(H)$ is the space of all orthonormal bases in $H$. The elements of a frame in any other orbit are therefore necessarily linearly dependent.

By definition, the action of $\GL(H)$ on each orbit in $F/\GL(H)$ is transitive. But because the elements of each orbit are frames, the action has even more structure, as the following lemma illustrates.
\begin{lemma} \label{Regular}
Consider any $f\in F$ and $A\in\GL(H)$. Then $Af=f$ if and only if $A=I$. In other words, the action of $\GL(H)$ is regular on each orbit in $F/\GL(H)$.
\end{lemma}
\begin{proof}
The reverse implication is trivial. For the forward implication, suppose $Af=f$. Recall $Af$ is defined by $(Af)(x) = A[f(x)]$ for all $x\in X$. Thus, $Af=f$ implies $A[f(x)] = f(x)$ for all $x\in X$. Since $f$ is a frame on $H$, the set $\{f(x):x\in X\}$ spans $H$. Since $A$ is linear on $H$, we have $A\phi = \phi$ for all $\phi\in H$. 
\end{proof}

Because the action of $\GL(H)$ is regular on each orbit in $F/\GL(H)$, every orbit is a principal homogeneous space. Therefore, the linear transformation connecting two frames is unique.

Lemma \ref{LinearDeform} implies that the frame operator map $S:F\mapsto \GL^+(H)$ may be thought of as a projection map, as the following proposition states.
\begin{prop}
The map $S$ is well-defined and surjective.
\end{prop}
\begin{proof}
The frame operator $S(f)$ of a frame $f$ is positive, bounded, and has a bounded inverse. Hence, $S$ is well-defined. Let $B\in\GL^+(H)$. Then, $B=AA^*$ for some $A\in\GL(H)$. Let $f_0\in F_0$ be a Parseval frame, and define $f=Af_0$. By Lemma \ref{LinearDeform}, $f$ is a frame and
\[ S(f) = S(Af_0)
= \ad_A(S(f_0))
= \ad_A(I)
= AA^*
= B. \]
Ergo, $S$ is surjective.
\end{proof}

We are now interested in showing that every frame can be transformed into a Parseval frame. Define the projection
\begin{equation}
T:F\mapsto F_0, \quad
T(f) = S(f)^{-\frac{1}{2}}f.
\end{equation}
The following proposition verifies that $T$ can indeed be thought of as a projection map.
\begin{prop}
The map $T$ is well-defined and surjective.
\end{prop}
\begin{proof}
Let $f\in F$. By Lemma \ref{LinearDeform}, observe that
\[ S(T(f)) = S(S(f)^{-\frac{1}{2}}f)
= S(f)^{-\frac{1}{2}}S(f) S(f)^{-\frac{1}{2}}
= I. \]
Thus, $T(f)\in F_0$, and hence $T$ is well-defined. Note $T$ fixes $F_0$ pointwise: If $f\in F_0$, then $T(f)=I^{-\frac{1}{2}}f=f$. Thus, $T$ is surjective.
\end{proof}

Therefore, every frame can be linearly transformed into a Parseval frame. But we would like this transformation to be unique in some sense. In particular, we would like to index the orbits in 
$F/\GL(H)$ by a set of Parseval frames. We must therefore determine how the Parseval frames in a common orbit in $F/\GL(H)$ are related. We recall that $F_0\subset F$ is the space of Parseval frames and consider the action of $U(H)$ on $F_0$.
\begin{lemma} \label{UnitaryParseval}
Let $f\in F_0$ and $A\in\GL(H)$. Then, $Af\in F_0$ if and only if $A\in U(H)$.
\end{lemma}
\begin{proof}
First assume $Af\in F_0$. Then,
\[ S(Af) = \ad_A(S(f))
= \ad_A(I)
= AA^*
= I. \]
Hence, $A\in U(H)$. For the converse, suppose $A\in U(H)$. Then, $Af$ is a frame and
\[ S(Af) = AA^* = I, \]
so that $Af\in F_0$.
\end{proof}

Therefore, all Parseval frames in a common orbit in $F/\GL(H)$ are unitarily equivalent, and hence it is possible to linearly transform or ``project'' any frame to a Parseval frame that is unique up to unitary equivalence. Let $\Fbar_0$ be a fixed transversal of the orbit space $F_0/U(H)$, so that $\Fbar_0$ is a maximal set of unitarily inequivalent Parseval frames on $H$. 
Note $F_0 = U(H)\Fbar_0$. By Lemma \ref{UnitaryParseval}, the ``factorization'' of a Parseval frame in $F_0$ into a unitary operator in $U(H)$ and a Parseval frame in $\Fbar_0$ is unique. We therefore define the projection maps
\begin{equation} \label{Usigma}
U:F_0\mapsto U(H) \mbox{ and } \sigma:F_0\mapsto \Fbar_0
\mbox{ such that } f = U(f)\sigma(f) \mbox{ for all } f\in F_0.
\end{equation}
We observe that for all $A\in U(H)$ and $f\in\Fbar_0$, we have $U(Af)=A$ and $\sigma(Af)=f$. Thus, $U$ and $\sigma$ are both surjective.

We are ready to show that $\Fbar_0$ indexes the orbits of $F/\GL(H)$. First, we define the maps
\begin{align}
\zeta:\GL(H)\times\Fbar_0\mapsto F, \quad
&\zeta(A, f) = Af \\
\zeta^+:\GL^+(H)\times F_0\mapsto F, \quad
&\zeta^+(A, f) = Af.
\end{align}
The maps $\zeta$ and $\zeta^+$ are instrumental to what follows, in large part due to the following theorem:
\begin{thm} \label{zeta}
The maps $\zeta$ and $\zeta^+$ are continuous bijections.
\end{thm}
\begin{proof}
First, we prove $\zeta$ is a bijection: Let $f\in F$. Since $T(f)\in F_0$, then $T(f)$ has the unique factorization $T(f) = U(T(f))\sigma(T(f))$. Note that $S(f)^{\frac{1}{2}}U(T(f))\in \GL(H)$ and $\sigma(T(f))\in \Fbar_0$. We have
\begin{align*}
\zeta(S(f)^{\frac{1}{2}}U(T(f)), \sigma(T(f)))
&= S(f)^{\frac{1}{2}} U(T(f))\sigma(T(f)) \\
&= S(f)^{\frac{1}{2}} T(f) \\
&= S(f)^{\frac{1}{2}} S(f)^{-\frac{1}{2}} f \\
&= f.
\end{align*}
Thus, $\zeta$ is surjective.

Suppose $\zeta(A_1,f_1) = \zeta(A_2,f_2)$. Then, $A_1f_1=A_2f_2$, and hence $(A_2^{-1}A_1)f_1=f_2$. Since $f_1$ and $f_2$ are Parseval, then Lemma \ref{UnitaryParseval} implies that $A_2^{-1}A_1$ is unitary. But since $f_1,f_2\in\Fbar_0$, then either $f_1$ and $f_2$ are unitarily inequivalent or $f_1=f_2$. Since $A_2^{-1}A_1\in U(H)$, then we must have $f_1=f_2$ and hence $A_2^{-1}A_1=I$ By Lemma \ref{Regular}. Thus, $A_1=A_2$. That is, $(A_1,f_1)=(A_2,f_2)$. Ergo, $\zeta$ is injective and therefore bijective.

Now, we prove $\zeta^+$ is a bijection: Let $f\in F$. Note $S(f)^{\frac{1}{2}}\in\GL^+(H)$ and $T(f)\in F_0$. We have
\begin{align*}
\zeta^+(S(f)^{\frac{1}{2}}, T(f))
&= S(f)^{\frac{1}{2}} T(f) \\
&= S(f)^{\frac{1}{2}} S(f)^{-\frac{1}{2}} f \\
&= f.
\end{align*}
Thus, $\zeta^+$ is surjective.

Suppose $\zeta^+(A_1,f_1)=\zeta^+(A_2,f_2)$. Thus, $A_1f_1=A_2f_2$, so that $(A_2^{-1}A_1)f_1=f_2$. Since $f_1$ and $f_2$ are Parseval, then Lemma \ref{UnitaryParseval} implies that $A_2^{-1}A_1$ is unitary. Thus,
\[ (A_2^{-1}A_1)(A_2^{-1}A_1)^*
= A_2^{-1}A_1A_1^*(A_2^{-1})^*
= I. \]
Since $A_2^{-1}$ and $A_1$ are positive and hence self-adjoint, then
\begin{align*}
A_2^{-1}A_1A_1A_2^{-1} &= I. \\
A_1^2 &= A_2^2.
\end{align*}
Since $A_1$ and $A_2$ are positive, then the unique principal square roots of $A_1^2$ and $A_2^2$ are precisely $A_1$ and $A_2$ respectively. Thus, we have $A_1=A_2$. This implies $A_2^{-1}A_1=I$, so that $f_1=f_2$. That is, $(A_1,f_1)=(A_2,f_2)$. Thus, $\zeta^+$ is injective and hence bijective.

Finally, we prove $\zeta$ and $\zeta^+$ are both continuous: Since $\zeta$ and $\zeta^+$ are both restrictions of the map $\zeta_*:\GL(H)\times F_0\mapsto F$, then it suffices to show $\zeta_*$ is continuous. Let $\{(A_n,f_n)\}_{n=1}^\infty$ be a sequence of points in $\GL(H)\times F_0$., and suppose $(A_n,f_n)\rightarrow (A,f)$. This means $A_n\rightarrow A$ and $f_n\rightarrow f$. Let $\eps>0$. Then, there exists $N_1\in\mathbb{N}$ such that $n>N_1$ implies
\[ \|A_n-A\|_H,\|f_n-f\|_J <\frac{\eps}{2\|A\|_H+\|f\|_J}. \]
Since $A_n\rightarrow A$, then there exists $N_2\in\mathbb{N}$ such that $n>N_2$ implies $\|A_n\|_H<2\|A\|_H$. Assume $n>\max\{N_1,N_2\}$. Then, we have
\begin{align*}
\|\zeta_*(A_n,f_n)-\zeta_*(A,f)\|_J &= \|A_nf_n-Af\|_J \\
&= \|A_nf_n-A_nf+A_nf-Af\|_J \\
&\leq \|A_nf_n-A_nf\|_J+\|A_nf-Af\|_J \\
&= \sup_{x\in X} \|A_n(f_n-f)(x)\|_H+\sup_{x\in X} \|(A_n-A)f(x)\|_H \\
&\leq \|A_n\|_H \sup_{x\in X}\|(f_n-f)(x)\|_H + \|A_n-A\|_H\sup_{x\in X}\|f(x)\|_H \\
&= \|A_n\|_H\|f_n-f\|_J+\|A_n-A\|_H\|f\|_J \\
&< 2\|A\|_H\left(\frac{\eps}{2\|A\|_H+\|f\|_J}\right)+\left(\frac{\eps}{2\|A\|_H+\|f\|_J}\right)\|f\|_J \\
&= \eps.
\end{align*}
Ergo, $\zeta_*$ and thus $\zeta$ and $\zeta^+$ are continuous.
\end{proof}

Because $\zeta:\GL(H)\times\Fbar_0\mapsto F$ is a bijection, the orbit space $F/\GL(H)$ is in one-one correspondence with $\Fbar_0$. In other words, the transversal $\Fbar_0$ of unitarily inequivalent Parseval frames indexes the orbits in $F$ induced by invertible linear transformations. In particular, because $\zeta$ is invertible, we have that 
$$F = \GL(H)\Fbar_0,$$ 
with every frame having a unique representation in $\GL(H)\Fbar_0$. Recalling the relationship between $\GL^+(H)$ and $\GL(H)$, the following corollary completes this line of thought.
\begin{cor} \label{Factorization}
We have
\[ F = \GL^+(H) U(H) \Fbar_0
= \GL(H) \Fbar_0
= \GL^+(H) F_0. \]
Moreover, the factorization of a frame $f\in F$ in $\GL^+(H)U(H)\Fbar_0$ as
\[ f = S(f)^{\frac{1}{2}}U(f)\sigma(T(f)) \]
is unique.
\end{cor}

Finally, we define the continuous projection maps
\begin{align*}
\pi_1:\GL(H)\times F_0\mapsto \GL(H), \quad
&\pi_1(A, f) = A \\
\pi_2:\GL(H)\times F_0\mapsto F_0, \quad
&\pi_2(A, f) = f.
\end{align*}
The relationships presented in this section can then be summarized by the following commuting diagram:

\begin{center}
\begin{tikzpicture}[every node/.style={midway}]
\matrix[column sep={3em}, row sep={3em}] at (0,0)
{
& \node(a){$\GL(H)/U(H)$}; & \\
\node(b){$\GL(H)$}; & & \node(c){$\GL^+(H)$}; \\
\node(d){$\GL(H)\times\Fbar_0$}; & \node(e){$F$}; & \node(f){$\GL^+(H)\times F_0$}; \\
\node(g){$\Fbar_0$}; & & \node(h){$F_0$}; \\
& \node(i){$F/U(H)$}; & \\
};
\draw[->] (b)--(a);
\draw[<->] (a)--(c);
\draw[->] (b)--(c) node[anchor=south]{$\rho$};
\draw[->] (d)--(b) node[anchor=east]{$\pi_1$};
\draw[->] (d)--(g) node[anchor=east]{$\pi_2$};
\draw[<->] (d)--(e) node[anchor=south]{$\zeta$};
\draw[<->] (e)--(f) node[anchor=south]{$\zeta^+$};
\draw[->] (f)--(c) node[anchor=south]{$S$};
\draw[->] (f)--(h) node[anchor=south]{$T$};
\draw[->] (f)--(c) node[anchor=west]{$\pi_1^2$};
\draw[->] (f)--(h) node[anchor=west]{$\pi_2$};
\draw[->] (h)--(g) node[anchor=south]{$\sigma$};
\draw[->] (h)--(i);
\draw[<->] (g)--(i);
\end{tikzpicture}
\end{center}

\noindent By $\pi_1^2$, we mean $\pi_1^2(A, f) = A^2$. Also, we have the following identity.
\begin{cor} \label{Identity}
For all $f\in F$, we have
\[ f = \zeta(\pi_1(\zeta^{-1}(f)), \pi_2(\zeta^{-1}(f))). \]
\end{cor}
\noindent In the next section, we extend these orbit space ideas into the structure of a fiber bundle.

\section{A Fiber Bundle of Frames}

We have seen that the space of frames $F$ may be fibrated into orbits that are principal homogeneous spaces under the action of $\GL(H)$. We have also seen that every orbit may be projected to a unique element in the transversal $\Fbar_0$ of unitarily inequivalent Parseval frames. We might therefore suspect that $F$ has the structure of a principal fiber bundle. But we cannot conclude this immediately because we do not know whether $\zeta^{-1}$ is continuous. In this section, we provide sufficient conditions for $\zeta^{-1}$ to be continuous and hence for $F$ to be a principal fiber bundle.

We begin by stating the definition of a fiber bundle.

\begin{defn}
Let $E_1$ and $B$ be topological spaces. A topological space $E$ is a called \emph{fiber bundle} with base space $B$ and fiber $E_1$ if there exists a \emph{projection} or continuous surjection $\pi:E\mapsto B$ that satisfies the \emph{local triviality condition}: For every $x\in E$, there is an open neighborhood $U\subseteq B$ about $\pi(x)$ and a homeomorphism $\theta:\pi^{-1}(U)\mapsto U\times E_1$ such that
\[ \pi(x) = (\pi_U\circ\theta)(x), \quad\forall x\in \pi^{-1}(U), \]
where $\pi_U:U\times E_1\mapsto U$ is the natural projection from the product space $U\times E_1$ to the first factor $B$. If the fiber $E_1$ is a principal homogeneous space under the action of a group $G$, then $E$ is called a \emph{principal fiber bundle} with \emph{structure group} $G$.
\end{defn}

Thus, a fiber bundle is simply a space that is locally a product space. Every product space $E = B\times E_1$ is a fiber bundle with base space either $B$ or $E_1$. A less trivial example of a fiber bundle is the M\"{o}bius strip with base space the circle $S^1$ and fiber $[0, 1]$. For more information on fiber bundles, see \cite{steenrod1951topology}.

Proceeding, we fix some Parseval frame $f_{10}\in\Fbar_0$ and define the space
\begin{equation} \label{F1}
F_1 = \GL(H)f_{10}
= \{Af_{10}: A\in \GL(H)\}.
\end{equation}
This space will ultimately be a fiber of $F$.

Our first task is to show that $F$ is in one-one correspondence with the product space $F_1\times\Fbar_0$. This means we have projection maps from $F$ to each component space $F_1$ and $\Fbar_0$. We already know that the map $\sigma\circ T$ projects $F$ onto $\Fbar_0$. Consequently, we define the projection map
\begin{equation} \label{T1}
T_1:F\mapsto F_1, \quad
T_1(f) = \pi_1(\zeta^{-1}(f))f_{10}.
\end{equation}
We can now show the following:
\begin{prop} \label{T1Surjection}
The map $T_1$ is surjective.
\end{prop}
\begin{proof}
Let $f_1\in F_1$. By Corollary \ref{Identity}, we have
\begin{align*}
f_1 &= \zeta(\pi_1(\zeta^{-1}(f_1)), \pi_2(\zeta^{-1}(f_1))) \\
&= \zeta(\pi_1(\zeta^{-1}(f_1)), f_{10}) \\
&= \pi_1(\zeta^{-1}(f_1)) f_{10} \\
&= T_1(f_1).
\end{align*}
\end{proof}

Since $F_1$ is a principal homogeneous space under the action of $\GL(H)$, it follows that $F_1$ and $\GL(H)$ are in one-one correspondence. Next, we define the map
\begin{equation}
\theta:F_1\mapsto\GL(H), \quad
\theta(f) = \pi_1(\zeta^{-1}(f)).
\end{equation}
This leads immediately to the following lemma:
\begin{lemma} \label{theta}
The map $\theta$ is a bijection.
\end{lemma}
\begin{proof}
Let $A\in\GL(H)$. Then, $Af_{10}\in F_1$. By Corollary \ref{Identity}, we have
\begin{align*}
Af_{10} &= \zeta(\pi_1(\zeta^{-1}(Af_{10})), \pi_2(\zeta^{-1}(Af_{10}))) \\
&= \zeta(\theta(Af_{10})), f_{10}) \\
&= \theta(Af_{10})f_{10}.
\end{align*}
But since $\GL(H)$ acts regularly on $F_1$ (by Lemma \ref{Regular}), then $\theta(Af_{10})=A$. Ergo, $\theta$ is surjective.

Suppose $\theta(f_1) = \theta(f_2)$. As above, $f_1$ and $f_2$ have the unique factorizations $f_1=\theta(f_1)f_{10}$ and $f_2=\theta(f_2)f_{10}$. But since $\theta(f_1)=\theta(f_2)$, then
\[ f_1 = \theta(f_1)f_{10} = \theta(f_2)f_{10} = f_2. \]
Thus, $\theta$ is injective.
\end{proof}

The bijection $\theta$ may be lifted to the map
\begin{equation}
\theta_*:F_1\times\Fbar_0\mapsto F, \quad
\theta_*(f_1, f_0) = \zeta(\theta(f_1), f_0).
\end{equation}
This leads to the following: 
\begin{thm}
The map $\theta_*$ is a bijection and has inverse
\[ \theta^{-1}(f) = (T_1(f), \sigma\circ T(f)). \]
\end{thm}
\begin{proof}
By Lemma \ref{theta}, $\theta$ is bijective. The identity map is obviously bijective. Thus, the map
\[ (f_1,f_0)\rightarrow (\theta(f_1), f_0) \]
is a bijection from $F_1\times\Fbar_0$ to $\GL(H)\times\Fbar_0$. By Theorem \ref{zeta}, $\zeta$ is bijective. Ergo, $\theta_*$ is a bijection.

To verify that the expression $\theta_*^{-1}$ is indeed the inverse of $\theta_*$, let $f\in F$ and consider
\begin{align*}
\theta_*(\theta_*^{-1}(f))
&= \zeta(\theta(T_1(f)), \sigma\circ T(f)) \\
&= \zeta[\pi_1\circ\zeta^{-1}(\pi_1\circ\zeta^{-1}(f)f_{10}), \sigma\circ T(f)] \\
&= \zeta(\pi_1\zeta^{-1}(f), \sigma\circ T(f)) \\
&= \pi_1(\zeta(f))\sigma(T(f)) \\
&= f.
\end{align*}
The reverse composition proceeds similarly.
\end{proof}

We therefore have the following commuting diagram:
\begin{center}
\begin{tikzpicture}[every node/.style={midway}]
\matrix[column sep={3em}, row sep={3em}] at (0,0)
{
\node(a){$F_1\times\Fbar_0$}; & \node(b){$F$}; \\
\node(c){$F_1$}; & \node(d){$\GL(H)$}; \\
};
\draw[<->] (a)--(b) node[anchor=south]{$\theta_*$};
\draw[->] (a)--(c);
\draw[->] (b)--(d) node[anchor=west]{$\pi_1\circ\zeta^{-1}$};
\draw[<->] (c)--(d) node[anchor=north]{$\theta$};
\draw[->] (b)--(c) node[anchor=south]{$T_1$};
\end{tikzpicture}
\end{center}

In particular, $F$ is in one-one correspondence with $F_1\times\Fbar_0$. But to show $F$ is a fiber bundle, we also require continuity. In particular, for $F$ to be a fiber bundle with base space $\Fbar_0$, the projection $\sigma\circ T = \pi_1\circ\zeta^{-1}$ mapping $F$ onto $\Fbar_0$ must be continuous. Since $\pi_1$ is continuous, it suffices to have $\zeta^{-1}$ be continuous.
\begin{prop} \label{IfzetaCont}
If $\zeta^{-1}$ is continuous, then $\theta_*:F_1\times\Fbar_0\mapsto F$ is a homeomorphism and $F$ is a principal fiber bundle with base space $\Fbar_0$, fiber $F_1$, and structure group $\GL(H)$.
\end{prop}
\begin{proof}
Suppose $\zeta^{-1}$ is continuous. Then, $\theta = \pi_1\zeta^{-1}$ is continuous. By Theorem \ref{zeta}, $\zeta$ is continuous. Thus, $\theta_*(f_1,f_0) = \zeta(\theta(f_1), f_0)$ is continuous

Since $\zeta^{-1}$ is continuous, then $T_1(f) = \pi_1(\zeta^{-1}(f))f_{10}$ and $\sigma\circ T = \pi_1\circ\zeta^{-1}$ are continuous. Thus, $\theta_*^{-1} = (T_1, \sigma\circ T)$ is continuous. Ergo, $\theta_*$ is a homeomorphism.

Since $F$ is homeomorphic to the product space $F_1\times\Fbar_0$ (via $\theta_*^{-1}$), then $F$ is trivially a fiber bundle as claimed.
\end{proof}

We therefore proceed to establish conditions that are sufficient for $\zeta^{-1}$ to be continuous. We first define the Banach space
\[ J_1 = L^1(X, H) \]
equipped with the norm
\[ \|f\|_{J_1} = \int_X \|f(x)\|_H\,d\mu(x). \]
Suppose that $F\subset J_1$. That is, suppose that all frames (in $J$) on $H$ are integrable. We will show that this is sufficient for $\zeta^{-1}$ to be continuous and hence for $F$ to be a fiber bundle.

By the commuting diagram in Section 2 and the unique factorization granted by Corollary \ref{Factorization}, it is straightforward to show that $\zeta^{-1}:F\mapsto\GL(H)\times\Fbar_0$ is given by
\begin{equation}
\zeta^{-1}(f) = (S(f)^{\frac{1}{2}}U(f), \sigma\circ T(f)).
\end{equation}
The three lemmas that follow show that each term on the left side of this equation is continuous in $J_1$.

\begin{lemma} \label{SCont}
The map $S:F\mapsto\GL^+(H)$ is continuous in the topology of $J_1$.
\end{lemma}
\begin{proof}
Let $\{f_n\}_{n=1}^\infty$ be a sequence of frames in $F$ and $f\in F$ such that $f_n\rightarrow f$ in $J_1$. Let $\eps>0$. Then, there exists $N\in\NN$ such that $n>N$ implies $\|f_n\|_{J_1}<2\|f\|_{j_1}$ and
\[ \|f_n-f\|_{J_1} < \frac{\eps}{3\|f\|_{J_1}}. \]
Suppose $n>N$. Consider any $\phi\in H$. We have
\begin{align*}
\|S(f_n)\phi-S(f)\phi\|_H
&= \left\|\int_X\brak{\phi, f_n(x)}_H f_n(x)\,d\mu(x) - \int_X\brak{\phi, f(x)}_H f(x)\,d\mu(x)\right\| \\
&= \left\|\int_X \brak{\phi, f_n(x)}_H f_n(x)\,d\mu(x) - \int_X\brak{\phi, f_n(x)}_H f(x)\,d\mu(x)\right. \\
& \qquad + \left.\int_X\brak{\phi, f_n(x)}_H f(x)\,d\mu(x) - \int_X\brak{\phi, f(x)}_H f(x)\,d\mu(x)\right\| \\
&= \left\|\int_X\brak{\phi, f_n(x)}_H [f_n(x)-f(x)]\,d\mu(x) + \int_X\brak{\phi, f_n(x)-f(x)}_H f(x)\,d\mu(x)\right\| \\
& \leq \int_X \|\phi\|_H\|f_n(x)\|_H\|f_n(x)-f(x)\|_H\,d\mu(x)  \\ 
& \qquad + \int_X\|\phi\|_H\|f_n(x)-f(x)\|_H\|f\|_H\,d\mu(x) \\
&= \|\phi\|\int_X \|f_n(x)-f(x)\|_H (\|f_n(x)\|_H+\|f(x)\|_H)\,d\mu(x).
\end{align*}
By H\"{o}lder's Inequality, we have
\begin{align*}
\|S(f_n)\phi-S(f)\phi\|_H
&\leq \|\phi\|_H\int_X\|f_n(x)-f(x)\|_H\,d\mu(x)\int_X(\|f_n(x)\|_H+\|f(x)\|_H)\,d\mu(x) \\
&= \|\phi\|_H\|f_n-f\|_{J_1}(\|f_n\|_{J_1}+\|f\|_{J_1}) \\
&< 3\|f\|_{J_1}\|f_n-f\|_{J_1}\|\phi\|_H \\
&< \eps \|\phi\|_H.
\end{align*}
Since this holds for all $\phi\in H$, then $\|S(f_n)-S(f)\|_H<\eps$. Ergo, $S(f_n)\rightarrow S(f)$, and hence $S$ is continuous.
\end{proof}

\begin{lemma} \label{TCont}
The map $T:F\mapsto F_0$ is continuous in the topology of $J_1$.
\end{lemma}
\begin{proof}
Let $\{f_n\}_{n=1}^\infty$ be a sequence of frames in $F$ and $f\in F$ such that $f_n\rightarrow f$ in $J_1$. By Lemma \ref{SCont}, $S(f_n)\rightarrow S(f)$. Since the map sending an operator in $\GL(H)$ to its inverse and the map sending an operator in $\GL^+(H)$ to its principal square root are both continuous in the operator norm, then $S(f_n)^{-\frac{1}{2}}\rightarrow S(f)^{-\frac{1}{2}}$. Let $\eps>0$. Then, there exists $N\in\NN$ such that $n>N$ implies $\|f_n\|_{J_1}<2\|f\|_{J_1}$ and
\[ \|f-f_n\|_{J_1},\|S(f_n)^{-\frac{1}{2}}-S(f)^{-\frac{1}{2}}\|_H < \delta=\frac{\eps}{2\|f\|_{J_1}+\|S(f)^{-\frac{1}{2}}\|_H}. \]
Suppose $n>N$. We have
\begin{align*}
\|T(f_n)-T(f)\|_{J_1}
&= \|S(f_n)^{-\frac{1}{2}}f_n-S(f)^{-\frac{1}{2}}f\|_{J_1} \\
&= \|S(f_n)^{-\frac{1}{2}}f_n-S(f)^{-\frac{1}{2}}f_n+S(f)^{-\frac{1}{2}}f_n-S(f)^{-\frac{1}{2}}f\|_{J_1} \\
&= \|[S(f_n)^{-\frac{1}{2}}-S(f)^{-\frac{1}{2}}]f_n+S(f)^{-\frac{1}{2}}(f_n-f)\|_{J_1} \\
&\leq \|S(f_n)^{-\frac{1}{2}}-S(f)^{-\frac{1}{2}}\|_H\|f_n\|_{J_1}+\|S(f)^{-\frac{1}{2}}\|_H\|f_n-f\|_{J_1} \\
&< \delta 2\|f\|_{J_1} + \|S(f)^{-\frac{1}{2}}\|_H\delta \\
&= \eps.
\end{align*}
Ergo, $T(f_n)\rightarrow T(f)$, and hence $T$ is continuous.
\end{proof}

Suppose $\Fbar_0$ is a continuous transversal of $F/\GL(H)$. That is, suppose that the projection $\sigma:F_0\mapsto\Fbar_0$ is continuous in the topology of $J_1$. Then we have the following:
\begin{lemma} \label{UCont}
The map $U:F_0\mapsto U(H)$ is continuous in the topology of $J_1$.
\end{lemma}
\begin{proof}
Let $f\in F_0$. Then, $f_0$ has the unique factorization $f = U(f)\sigma(f)$. For any $\phi\in H$ and for all $x\in X$, we have
\[ \brak{\phi, f(x)}_H = \brak{\phi, U(f)\sigma(f)(x)}_H
= \brak{U(f)^*\phi, \sigma(f)(x)}_H. \]
Letting $V_g$ and $V_g^*$ denote the analysis and synthesis operators of a frame $g\in F$, we have
\[ V_f\phi = V_{\sigma(f)}U(f)^*\phi, \]
or simply $V_f = V_{\sigma(f)}U(f)^*$. Thus, $V_fU(f) = V_{\sigma(f)}$, and hence $V_f^*V_fU(f) = V_f^*V_{\sigma(f)}$. But since $f$ is Parseval, then $V_f^*V_f = S(f)=I$ so that
\[ U(f) = V_f^* V_{\sigma(f)}. \]
Thus, for all $\phi\in H$, we have
\begin{equation}
U(f)\phi = \int_X\brak{\phi, \sigma(f)(x)}_H f(x)\,d\mu(x).
\end{equation}

To show $U$ is continuous, let $\{f_n\}_{n=1}^\infty$ be a sequence of frames in $F_0$ and $f\in F_0$ with $f_n\rightarrow f$ in $J_1$. Since $\sigma$ is continuous, then $\sigma(f_n)\rightarrow\sigma(f)$. Let $\eps>0$. Then, there exists $N\in\NN$ such that $n>N$ implies $\|f_n\|_{J_1}<2\|f\|_{J_1}$ and
\[ \|f_n-f\|_{J_1}, \|\sigma(f_n)-\sigma(f)\|_{J_1} < \delta = \frac{\eps}{\|\sigma(f)\|_{J_1}+2\|f\|_{J_1}}. \]
Suppose $n>N$. Let $\phi\in H$. By manipulations similar to those used in the proof of Lemma \ref{SCont} (including the Triangle Inequality, Cauchy-Schwartz Inequality, and H\"{o}lder's Inequality), we have
\begin{align*}
\|U(f_n)\phi-U(f)\phi\|_H
&= \left\|\int_X\brak{\phi, \sigma(f_n)(x)}_H f_n(x)\,d\mu(x) - \int_X\brak{\phi, \sigma(f)(x)}_H f(x)\,d\mu(x)\right\| \\
&\leq \|\phi\|_H \|\sigma(f_n)-\sigma(f)\|_{J_1}\|f_n\|_{J_1} + \|\phi\|_H \|\sigma(f)\|_{J_1} \|f_n-f\|_{J_1} \\
&< \|\phi\|_H \delta 2\|f\|_{J_1} + \|\phi\|_H \|\sigma(f)\|_{J_1} \delta \\
&= (2\|f\|_{J_1}+\|\sigma(f)\|_{J_1})\delta \|\phi\|_H \\
&= \eps \|\phi\|_H.
\end{align*}
Since this holds for all $\phi\in H$, then $\|U(f_n)-U(f)\|_H < \eps$. Ergo, $U(f_n)\rightarrow U(f)$, and hence $U$ is continuous.
\end{proof}

\noindent The lemmas \ref{SCont}-\ref{UCont} thus lead to the following theorem.
\begin{thm} \label{FiberBundle}
The map $\zeta^{-1}$ is continuous in the topology of $J_1$. Moreover, $F$ is a principal fiber bundle with base space $\Fbar_0$, fiber $F_1$, and structure group $\GL(H)$ in the topology of $J_1$.
\end{thm}
\begin{proof}
Recall that $\zeta^{-1}$ is given by
\[ \zeta^{-1}(f) = (S(f)^{\frac{1}{2}}U(f), \sigma\circ T(f)). \]
By Lemmas \ref{SCont}-\ref{UCont}, the maps $S$, $T$, and $U$ are continuous in $J_1$. Moreover, $\sigma$ and the square root function are continuous as well. Therefore, $\zeta^{-1}$ is continuous. By Proposition \ref{IfzetaCont}, $F$ is a fiber bundle as claimed.
\end{proof}

A special case occurs when $\mu$ is a finite measure on $X$. In this case, convergence in $J=L^\infty(X,H)$ implies convergence in $J_1=L^1(X,H)$. We therefore have the following corollary.

\begin{cor} \label{FiniteMeasure}
If $\mu$ is a finite measure on $X$, then the space $F$ is a principal fiber bundle in the topologies of both $J$ and $J_1$.
\end{cor}

\section{Applications, Examples, and Future Directions}

Although typically the group $\GL(H)$ is quite large, it is nonetheless independent of the indexing set $X$.  Thus, Theorem \ref{FiberBundle} is a description of the structure created by frames as maps from $X$ into $H.$   We believe that it has the potential to produce many applications and future directions of study.  

For example, when $X$ is discrete and $\mu$ is a finite counting measure, we obtain the following corollary.
\begin{cor}
Let $H$ be finite-dimensional and $X$ be a finite set with counting measure $\mu$. Then the space $F$ of all finite frames on $H$ indexed by $X$ is a principal fiber bundle with base space $\Fbar_0$, fiber $F_1$, and structure group $\GL(H)$.
\end{cor}

This corollary is important because it suggests meaningful generalizations of the singular value decomposition of a matrix.  For example, if $X=\left\{ 1,\ldots ,n\right\} $ where $n\geq m>0,$ then we will say that an $n\times m$ matrix $T$ is a \emph{frame matrix}, or simply a \emph{frame} if the context is understood, if the set of columns obtained by Hermitian transposing its rows are a frame.  Moreover, given a frame $\left\{ f_{1},\ldots,f_{n}\right\} \subset \mathbb{C}^{m},$ we will say that it generates the frame matrix $T$ if the rows of $T$ are  $f_{j}^{\ast },$ $j=1,\ldots ,n.$ 

A frame matrix has a singular value decomposition of 
\[
T=U\tilde{\Sigma}V^{\ast }\ \ where\ \ \tilde{\Sigma}=
\left[ \begin{array}{c}
           \Sigma  \\ 
             0
       \end{array}
\right] 
\]
Moreover, $\tilde{\Sigma}$ is itself a frame, and if $\Sigma =I_{m},$ then $\tilde{\Sigma}$ is also a Parseval frame. 

However, not all Parseval frames are unitarily equivalent (in $\mathbb{C}^{m}$) to $\tilde{\Sigma}.$ Instead, if we define $M_{\Sigma }$ to be the set of $n\times m$ matrices whose rows
have at most a single non-zero entry and whose columns are unit vectors in $\mathbb{C}^{n}$, then a choice of base space $\bar{F}_{0}$ for $X=\left\{1,\ldots n\right\} $ over $\mathbb{C}^{m}$ is a maximally non-unitary subset of $M_{\Sigma }$ (there are matrices in $M_{\Sigma }$ that are unitarily equivalent).  It follows that every Parseval frame over $\mathbb{C}^{m}$ with $n$ elements  is of the form 
\[
F=TV^{\ast },\ \ \ T\in M_{\Sigma }
\]
and similarly, every frame with $n$ elements is of the form $TG,$ where $T\in M_{\Sigma }$ and $G\in GL\left( \mathbb{C}^{m}\right) .$ Similar extensions follow immediately.

The above example reveals that the base space $\Fbar_0$ can be obtained by the action of unitary operators $V^*$ on the frame $T$. Note that the operators $V^*$ do not act on the space $\CC^m$ that the frame spans but rather the possibly larger space $\CC^n$. In the next paper, we extend this idea to show how the elements of a base space $\Fbar_0$ are connected to one another and therefore how one can ``move'' between fibers in a fiber bundle of frames. We also do this in the case of $X$ being non-countable, with applications to special classes of continuous
frames such as coherent states, Gabor Frames, and wavelets.

We will reveal the structure of $\Fbar_0$ by combining the results obtained in this paper with the structure of reproducing kernel Hilbert spaces, which are Hilbert spaces in which a function can be evaluated (or reproduced) by integrating it against a kernel function. If the kernel satisfies a certain condition, then this reproducing property is equivalent to the reconstruction property of frames \cite{ali2014coherent}. The connection of frames to reproducing kernel Hilbert spaces has applications to machine learning. In particular, because frames are more general than bases, then frames can be used to construct a variety of kernel functions for use in algorithms such as support vector machines \cite{rakotomamonjy2005frames}.

The reproducing kernel results can then be applied to Gabor frames in the sense that they are frames of coherent states indexed by a reproducing kernel Hilbert space \cite{ali2014coherent, heil1989continuous}. We will show that there is an accompanying correspondence between nonlinear deformations of frames and linear maps between reproducing kernel Hilbert spaces. In particular, we show that all Parseval frames in a Hilbert space are connected by transformations that are unitary between reproducing kernel Hilbert spaces. We therefore establish the structure that connects all frames on a Hilbert space-- namely, transformations that are linear in a ``larger'' space. We also provide conditions under which a linear transformation between reproducing kernel Hilbert spaces may be pulled back directly to a deformation of frames. We apply our findings to the example of frame discretization in order to view discretization as a true frame-deforming operator.


\bibliographystyle{plain}
\bibliography{FramePaperBib}

\end{document}